\documentclass{article}
\usepackage{fullpage}
\usepackage{amssymb}
\usepackage{amsmath}
\usepackage{amsthm}
\usepackage{xcolor}
\usepackage{xargs}
\usepackage{epigraph}
\usepackage{tikz}
\usetikzlibrary{calc}
\theoremstyle{definition}

\newtheorem{conjecture}{Conjecture}
\newtheorem{construction}{Construction}

\theoremstyle{plain}

\newtheorem{theorem}[construction]{Theorem}
\newtheorem{corollary}[construction]{Corollary}
\newtheorem{lemma}[construction]{Lemma}

\theoremstyle{remark}

\newcommand{\Deg}[2]%
{%
  \ensuremath{\textsc{Deg}_{#1}{\left(#2\right)}}%
}
\newcommand{\dist}[2]%
{%
  \ensuremath{\textsc{Dist}_{#1}{\left(#2\right)}}%
}

\definecolor{MTUGOLD}{RGB}{255,207,0}
\definecolor{UMDMAROON}{RGB}{109,0,19}

\title{A note on nearly platonic graphs}

\author{Dalibor Froncek, William Keith, Donald L. Kreher}

\begin{document}

\maketitle
\begin{abstract}
A nearly platonic graph is a $k$-regular simple planar graph in which all but a small number of the faces have the same degree.  We show that it is impossible for a finite graph to have exactly one disparate face, and offer some conjectures, including the conjecture that graphs with two disparate faces come in a small set of families.
\end{abstract}

\section{Introduction}

Several authors (\cite{Crowe}, \cite{HorJuc}, \cite{Jen}, \cite{Malkevitch}) have been interested in planar embeddings of graphs in which almost all faces are of one type, with one or two exceptions. For the most part, these papers deal with \emph{nearly regular} planar graphs: those in which most faces and vertices are of degrees that are a \emph{multiple} of some $m$, and a small number of other faces have degrees that are not a multiple of $m$. The proof techniques involve transformations which may change the number of edges of one or more faces, preserving divisibility of their degrees by $m$.  A typical theorem in the area is Lemma 2.2 of \cite{Malkevitch}, which states that no 3-regular planar graph exists in which all but one face has degree a multiple of three.

These theorems thus leave open the full question with which this article is interested: is it possible to produce a vertex-regular planar graph in which almost all faces have one degree and a small number of faces have a different degree, regardless of whether the disparate face degrees are multiples of some $m$ -- e.g., can a 3-regular graph be drawn in which all faces are triangles except for a single 9-gon -- and if so, what restrictions exist on the construction?\footnote{For the interested professor, the question arose in the context of teaching an introductory combinatorics course, in an attempt to construct a graph with exceptional outer face in anticipation of student error.}

For a single exceptional face, the answer is in the negative: it is not possible to find a finite, planar, regular graph that has all but one face of one degree and a single face of a different degree.  For two exceptional faces, all of our constructions at present are simple variants of the Platonic graphs; we conjecture that these are the only possibilities.  For three exceptional faces, constructions become abundant.

For a question so easily stated, one suspects that the result is already folklore, perhaps demanding greater than usual diligence in checking the literature.  However, a search through the standard graph theory textbooks yields no relevant theorem, a query on MathOverflow (\cite{MOquestion}) attracted no firm answer, citations of \cite{Crowe}--\cite{Malkevitch} remain interested in nearly regular graphs, and plausible variations of the name ``nearly regular'' garnered no relevant papers.  Although the only theorems we require are basic theorems of graph theory, the case by case vertex-counting required is sufficiently delicate that we now have some confidence the theorem has not been previously published.

In the next subsection we recall the relevant theorems of graph theory and construct the basic properties we will make use of in the sequel.  In Section 2 we establish the negative answer for the case with a single exceptional face; in the final section we discuss the cases of two and three exceptional faces, and offer some open questions.

\subsection{Basic theorems}

A $(v,e,f)$-graph will denote a graph with that has $v$ vertices and $e$ edges that has a planar embedding with $f$ faces. Consider such a graph in which the degree of each vertex is $k$, there are $f_1$  faces of degree  $d_1$, and the remaining $f_2=f-f_1$ have degree $d_2$.  Every edge has two ends and abuts two faces, so twice the number of edges must equal both the sum of the degrees of all the vertices, and the sum of the degrees of the faces:

\begin{align*}
2e&=kv\\
2e&=f_1d_1+f_2d_2.
\end{align*}

\noindent An important theorem in graph theory is Euler's formula, which holds that for all planar graphs,
\begin{align*}
v-e+f = 2.
\end{align*}

\noindent Putting these pieces together and solving for various values we obtain:
\begin{align}
f&=\frac{kv-f_1(d_1-d_2)}{d_2}\label{f}\\
v(2d_2-k d_2+2k)&= 2f_1 d_1 + (4-2f_1) d_2 \label{v}\\
\frac{e}{kd_2}\biggl(4-(k-2)(d_2-2)\biggr)&=
\Phi(f_1,d_1,d_2),\label{Phi}
\end{align}
where
\[
\Phi(f_1,d_1,d_2)=2+\frac{f_1(d_1-d_2)}{d_2}
= 2+f_1\left(\frac{d_1}{d_2}-1\right).
\]

If $k=2$, our graph is just a polygon, which has two faces of equal degree (the inner and the outer).  Ignoring those, we have $3 \leq k \leq 5$, and $d_i \geq 3$, because faces must be at least triangles.

Now we can show that, regardless of $k$,

\begin{lemma} If $f_1 \leq 3$, then $\Phi(f_1,d_1,d_2) >0$.
\begin{proof} If $d_1 \geq d_2$, then obviously $\Phi(f_1,d_1,d_2)>0$, so we assume $d_1<d_2$. Then
\[
-1 < \frac{d_1-d_2}{d_2} <0
\]
and so
\[
2-f_1 < \Phi(f_1,d_1,d_2)<2.
\]
Hence if $f_1 \leq 2$, then $\Phi(f_1,d_1,d_2) > 0$. We now consider $f_1=3$. If $d_2 \geq 6$, then
\[
2e = d_2(f-3)+3d_1 \geq 6(f-3)+9 = 6f-9= 2(3f-6)+3\geq 2e+3
\]
a contradiction. Hence $d_2 \leq 5$. Then
\[
\Phi(3,d_1,d_2) \geq 2+3\left( \frac{3}{5}-1 \right)=
\frac{4}{5} >0.
\]
\end{proof}
\end{lemma}

\begin{corollary}\label{five}
If $f_1 \leq 3$, then
\begin{equation}
(k,d_2)= (3,3), (3,4), (3,5), (4,3),\text{ or } (5,3).
\end{equation}
\begin{proof}
The lemma shows that $\Phi(f_1,d_1,d_2)$ is positive, when $f_1\leq 3$.  Then Equation~\ref{Phi} forces
\[
(k-2)(d_2-2) < 4.
\]
There are only five integral solutions to this inequality when $k,d_2 \geq 3$. They are the solutions listed.
\end{proof}
\end{corollary}

If $f_1=0$ the five Platonic solids are obtained.  One corresponds to each of the 5 possibilities enumerated in Corollary~\ref{five}, and a little more work (see any relevant graph theory textbook, for instance \cite{GraverWatkins}) shows that these are the only possible such graphs.

\section{$f_1=1$: An Impossible Mistake}
If $f_1=1$, then a single face has degree different from all the others. Inspired by the adage of J\'{a}ra Cimrman
\begin{quotation}
Platonick\'{a} l\'{a}ska nem\r{u}\v{z}e b\'{y}t \v{c}\'{a}ste\v{c}n\'{a}, mus\'{\i} b\'{y}t \'{u}pln\'{a}
\end{quotation}
we show that such a graph cannot exist.

We study the five possibilities for $(k,d_2)$ above in turn.  In each case, we will calculate the allowable number of vertices as a function of $d_1$: substitute $k$, $d_2$ and $f_1$ into Equation \ref{v} and solve through for $v$.
\begin{equation}\label{v when f_1=1}
v=\frac{2(d_1+d_2)}{4-(k-2)(d_2-2)}
\end{equation}
Then we will consider how they might be adjacent to each other, eventually deriving a contradiction.  In each case, what we essentially show, reformulated, is that the face regularity requirement has to be weakened further to gain any new graphs: the class of possible graphs for a given $(k,d_2)$ with $f_1 \leq 1$ is still populated only by the Platonic graphs, $f_1 = 0$.

Without loss of generality we may assume that the graph has been drawn in the plane so that $F$, the unique face of degree $d_1$, is the outer face. Let $x_0 x_{1} \ldots x_{d_1-1} x_0$ be $\partial F$, the cycle bounding $F$, and denote by $\dist{F}{x_i,x_j}$ the length of the shortest $x_i$ to $x_j$ path on $\partial F$. All remaining vertices and edges are interior to $\partial F$. An edge that is not part of $F$'s bounding cycle, but joins two vertices of the cycle, is called a chord.

\begin{lemma}\label{nochord}
For $f_1 = 1$, $(k,d_2)\in\{(3,3),(3,4),(3,5),(4,3), (5,3)\}$, the outer face has no chords.
\end{lemma}

\begin{proof}
We proceed by contradiction.  Assume there exists such a graph with outer face $F$ and a chord.

Suppose without loss of generality that $x_0x_j$ is the chord.  If $k=3$, then obviously $3\leq j\leq d_{1}-3$, otherwise $x_1$ or $x_{d_1-1}$ is of degree 2. Let $y_i$ be the vertices within the region $R_1$ bounded by the cycle $x_0x_1 \dots x_j x_0$ and $z_i$ the vertices within the region $R_2$ bounded by $x_j x_{j+1} \dots x_0 x_j$.

\item[\boldmath $k=3, d_2=3$.]
Because $x_0$ is already of degree 3, the path $x_1,x_0,x_j$ must be on the boundary of a triangular face, forcing edge $x_1x_j$, which implies $\deg{x_j}\geq 4$, a contradiction.

\item[\boldmath $k=3, d_2=4$.]
We observe that because both $x_0$ and $x_j$ are already of degree 3, the path $x_1 x_0 x_j x_{j-1}$ must be on the boundary of a rectangular face, which forces edge $x_1x_{j-1}$. Then because both $x_1$ and $x_{j-1}$ are of degree 3, the path $x_2 x_1 x_{j-1} x_{j-2}$ must be on the boundary of a rectangular face, which forces edge $x_2x_{j-2}$. Continuing this way, we either form a rectangular face $x_{t-1} x_{t} x_{t+1} x_{t+2}$ when $j=2t+1$ -- but $x_{t}$ and $x_{t+1}$ are still just of degree 2, a contradiction -- or we form a triangular face $x_{t-1} x_{t} x_{t+1}$ when $j=2t$, a contradiction as well.  Any vertices internal to the final face cannot be adjacent to $x_{t-1}$ and $x_{t+2}$ or $x_{t+1}$ respectively, meaning the final face bounded by the last chord would be of degree greater than 4.

\item[\boldmath $k=3,d_2=5$.]

Without loss of generality, assume $x_0$ is adjacent to $x_j$ and $j$ is minimal in that no edge $x_\ell x_k$ exists with $0 \leq \ell, k < j$.  We produce the contradiction illustrated in Figure \ref{nochordk3d25}.

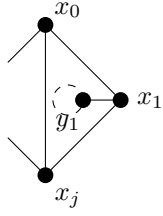
\begin{figure}\label{inflor}
\begin{center}\begin{tikzpicture}
\draw (-0.5,0.5) -- (0,0) -- (0,2) -- (1,1) -- (0,0);
\draw (-0.5, 1.5) -- (0,2);
\draw (1,1) -- (0.5,1);
\draw[dashed] (0.3,1) circle (0.2);
\fill (0,0) circle (3pt);
\fill (0,2) circle (3pt);
\fill (1,1) circle (3pt);
\fill (0.5,1) circle (3pt);
\draw (0.3,-0.3) node {$x_j$};
\draw (1.4, 1) node {$x_1$};
\draw (0.3, 2.2) node {$x_0$};
\draw (0.3,0.7) node {$y_1$};
\end{tikzpicture}
\end{center}
\caption{A basic inflorescence.}
\end{figure}

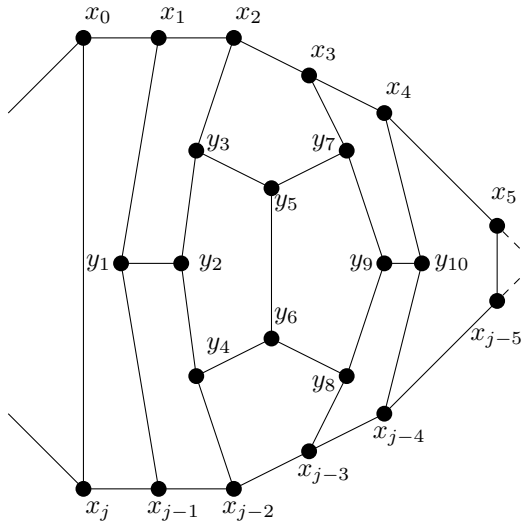
\begin{figure}\label{nochordk3d25}
\begin{center}
\begin{tikzpicture}
\draw (-1,1) -- (0,0) -- (1,0) -- (2,0) -- (3,0.5) -- (4,1)  -- (5.5,2.5);
\draw[dashed] (5.8,2.8) -- (5.5, 2.5);
\draw (0,0 -- (0,6);
\draw (-1,5) -- (0,6) -- (1,6) -- (2,6) -- (3,5.5) -- (4,5) -- (5.5,3.5);
\draw [dashed] (5.8,3.2) -- (5.5, 3.5);
\draw (0,0) -- (0,6);
\draw (0.2,-0.3) node {$x_j$};
\fill (0,0) circle (3pt);
\draw (0.2,6.3) node {$x_0$};
\fill (0,6) circle (3pt);
\draw (1.2,6.3) node {$x_1$};
\fill (1,6) circle (3pt);
\draw (2.2,6.3) node {$x_2$};
\fill (2,6) circle (3pt);
\draw (1.2,-0.3) node {$x_{j-1}$};
\fill (1,0) circle (3pt);
\draw (2.2,-0.3) node {$x_{j-2}$};
\fill (2,0) circle (3pt);
\draw (3.2,5.8) node {$x_3$};
\fill (3,5.5) circle (3pt);
\draw (3.2,0.2) node {$x_{j-3}$};
\fill (3,0.5) circle (3pt);
\draw (4.2,5.3) node {$x_4$};
\fill (4,5) circle (3pt);
\draw (4.2,0.7) node {$x_{j-4}$};
\fill (4,1) circle (3pt);
\draw (5.5,2) node {$x_{j-5}$};
\fill (5.5,2.5) circle (3pt);
\draw (5.6,3.9) node {$x_5$};
\fill (5.5,3.5) circle (3pt);
\draw (1,0) -- (0.5,3) -- (1,6);
\draw (0.2,3) node {$y_1$};
\fill (0.5,3) circle (3pt);
\draw (0.5,3) -- (1.3,3);
\draw (1.7,3) node {$y_2$};
\fill (1.3,3) circle (3pt);
\draw (2,6) -- (1.5,4.5) -- (1.3,3) -- (1.5,1.5) -- (2,0);
\draw (1.8,4.6) node {$y_3$};
\fill (1.5,4.5) circle (3pt);
\draw (1.8,1.9) node {$y_4$};
\fill (1.5,1.5) circle (3pt);
\draw (1.5,4.5) -- (2.5,4);
\draw (2.7,3.8) node {$y_5$};
\fill (2.5,4) circle (3pt);
\draw (1.5,1.5) -- (2.5,2);
\draw (2.7,2.3) node {$y_6$};
\fill (2.5,2) circle (3pt);
\draw (3,5.5) -- (3.5,4.5) -- (2.5,4) -- (2.5,2) -- (3.5,1.5) -- (3,0.5);
\draw (3.2,4.6) node {$y_7$};
\fill (3.5,4.5) circle (3pt);
\draw (3.2,1.4) node {$y_8$};
\fill (3.5,1.5) circle (3pt);
\draw (3.5,4.5) -- (4,3) -- (3.5,1.5);
\draw (3.7,3) node {$y_9$};
\fill (4,3) circle (3pt);
\draw (4,3) -- (4.5,3);
\draw (4.9,3) node {$y_{10}$};
\fill (4.5,3) circle (3pt);
\draw (4,5) -- (4.5,3) -- (4,1);
\draw (5.5,2.5) -- (5.5,3.5);
\end{tikzpicture}
\end{center}
\caption{Contradiction for $k=3$, $d_2 = 5$ boundary self-adjacency.}
\end{figure}

We have $x_1 \neq x_j$, $x_{j-1} \neq x_0$ to avoid a multigraph.  Likewise $1 \neq j-1$ else $x_1$ must be adjacent to $y_1$ not on the boundary of $F$, and $y_1$ must in turn be adjacent to some other vertices within this face, since the boundary vertices are all of degree 3 already.  But this makes $y_1 x_1$ a bridge, and the face within which it lies is of degree strictly greater than 5.  This is illustrated in figure \ref{inflor}.  Call such an instance an \emph{inflorescence} for the remainder of this argument.

So there are at least two distinct vertices $x_1$ and $x_{j-1}$.  Now $x_1$ must connect to some $y_1$ and $x_{j-1}$ to some $y_{j-1}$.  But then $y_1 = y_{j-1}$ to make the pentagon $y_1 x_1 x_0 x_j x_{j-1}$.  To give $y_1$ degree 3, it must be adjacent to some $y_2$; if it were adjacent to $x_2$ it would create a triangle, and to $x_3$ or higher a face of degree greater than 5, as $x_2$ would require an inflorescence.

Now $x_1 x_{j-1}$ is not an edge, else $y_1 y_2$ is an inflorescence causing a face of degree above 5, nor is $x_2 = x_{j-2}$, else either $x_2 y_2$ is an edge, creating at least one face of degree 4, or it is not an edge, in which case $x_2$ is adjacent to some $y_3$, which must be adjacent to $y_2$ to close two pentagonal faces, yet neither $y_2$ nor $y_3$ yet has degree 3, so inflorescences would increase the degree of one or both of the internal faces with $x_2$ on the boundary.

Now $y_2$ is not adjacent to $x_2$ or $x_{j-2}$ (square, or greater with inflorescence), so it must be adjacent to two $y_i$, say $y_3$ and $y_4$.  These must be adjacent to $x_2$ and $x_{j-2}$ to close the faces.  We have $x_2$ not adjacent to $x_{j-2}$, else $y_3$ is connected by a path of length 2 to $y_4$, say via $y_5$, forming a face of degree 4 or, with an inflorescence from $y_5$, degree 6 or more.

Neither $y_3$ nor $y_4$ can be adjacent to each other (a triangle is formed, or a face of degree greater than 5 with an inflorescence), nor by a path via a $y_5$ of length 2 (a square is formed, or a face of degree 6 or more); thus $y_3$ is adjacent to $y_4$ by a path of length 3, say via $y_5$ and $y_6$.  We cannot now have $x_3 = x_{j-3}$, since in such a case if $x_3$ is adjacent to $y_5$ or $y_6$, a face of degree 4 (or 6 or more) is formed, while if not adjacent to either, it must be adjacent to some $y_7$ which in turn is adjacent to both $y_5$ and $y_6$, forming a triangle.  So $x_3$ and $x_{j-3}$ exist and are distinct.

Now $x_3$ is not adjacent to $y_5$ (square), $y_6$ ($y_5$ would root an infloresence into a pentagon), so it is adjacent to some $y_7$, and $y_7$ must be adjacent to $y_5$ to close a face.  Likewise $x_{j-3}$ is adjacent to some $y_8$ in turn adjacent to $y_6$, with $y_7 \neq y_8$, else $y_6$ is on the boundary of a face of degree 6 or more.

We have $y_7$ not adjacent to $y_8$ (square), and so must be adjacent via a path of length 2; the intermediate vertex cannot be an $x_i$ since this would increase the degree of the vertex to 4 or more, so say the intermediate vertex is $y_9$.  We now have $x_3$ not adjacent to $x_{j-3}$, else $y_9$ would root an inflorescence, nor is $x_4 = x_{j-4}$, since if $y_9$ is adjacent to $x_4$ squares are created, and if not, the path from $y_9$ to $x_4$ would have at most one intermediate vertex which would root an inflorescence.

Now $y_9$ is not adjacent to $x_4$ or $x_{j-4}$ (square), so it must be adjacent to a $y_{10}$, which to close faces must in turn be adjacent to $x_4 $ and $x_{j-4}$.  But now to make a face of degree 5, both $x_5$ and $x_{j-5}$ must exist, but must be adjacent; but then any other path from $x_5$ to $x_{j-5}$, which must not include $x_4$, $x_{j-4}$ or $y_{10}$, will be part of the boundary of a face of degree greater than 5, a contradiction.

The required vertices are illustrated in Figure~\ref{nochordk3d25}.

\item[\boldmath $k=4, d_2=3$.]
Suppose that the fourth neighbor of $x_0$ is in $R_2$, that is, it is either $z_i$ or $x_i$ for $j+1\leq d_1-2$. Then since $x_0$ is already of degree 4, the path $x_1x_0x_j$ must be on the boundary of a triangular face, forcing edge $x_1x_j$. Now since $x_j$ is already of degree 4, the path $x_1,x_j,x_{j-1}$ must be on the boundary of a triangular face, forcing edge $x_1x_{j-1}$. Once more, $x_1$ is now of degree 4, so the path $x_2,x_1,x_{j-1}$ must be on the boundary of a triangular face, forcing edge $x_2x_{j-1}$. We continue until the forced edge reaches $x_{\lfloor{\frac{j}{2}}\rfloor}$ when $j$ is odd or $x_{\frac{j}{2}+1}$ when $j$ is even. Then there is only one vertex of degree 2 left on the boundary of $R_1$, namely $x_{\lceil{\frac{j}{2}}\rceil}$ when $j$ is odd or $x_{\frac{j}{2}}$ when $j$ is even, and the next forced edge would be a multiple edge, a contradiction.  The argument works in the opposite direction if the fourth neighbor of $x_0$ is in $R_1$.

\item[\boldmath $k=5,d_2=3$.]

Assume that $x_0 x_j$ is minimal in the sense that no chord $x_i x_k$ exists with $0 \leq i < k \leq j$ other than $x_0 x_j$ itself.  Since $k=5$, $x_0$ and $x_j$ both have two other neighbors.  Clearly if both neighbors of $x_0$ (resp. $x_j$) are within $R_2$, then in order to make a triangular face, we must have a chord $x_1 x_j$ (resp. $x_0 x_{j-1}$), a contradiction.  If both neighbors of both $x_0$ and $x_j$ are within $R_1$, then the path $x_{d-1} x_0 x_j x_{j+1}$ must border a face of degree at least 4, also a contradiction.

Thus, either $x_0$ and $x_j$ both have exactly one more neighbor in each of $R_1$ and $R_2$, or $x_0$ has both additional neighbors in $R_1$ and $x_j$ has exactly one additional neighbor in each of $R_1$ and $R_2$, or $x_0$ has one neighbor in each $R_i$ and $x_j$ has both neighbors in $R_1$.  The latter two are the same case after a relabeling, and so we deal with the former.

In both cases, the contradiction results from our conditions forcing the construction of the icosahedron; the minimality-contradicting edge is on the border of its planar embedding.

\phantom{.}

\noindent \textbf{Case 1:} Suppose both vertices have one neighbor in each region.  We produce the contradiction to minimality illustrated in Figure 3.


\begin{figure}\label{k5d23caseone}
\begin{center}
\begin{tikzpicture}
\draw (0,0) -- (0,4) -- (1,2) -- (0,0) -- (2,0) -- (1,2) -- (2,4) -- (3,2) -- (2,0) -- (4,1) -- (3,2) -- (4,3) -- (2,4) -- (5,4) -- (4,3) -- (5,2) -- (4,1) -- (5,0) -- (5,4) -- (6,4) -- (6,0) -- (2,0);
\draw (0,0) -- (-1,2) -- (0,4) -- (2,4);
\draw (1,2) -- (3,2);
\draw (4,1) -- (4,3);
\draw (5,0) -- (6,2) -- (5,4);
\draw (5,2) -- (6,2);
\draw [dashed] (6,0) -- (7,0);
\draw [dashed] (6,4) -- (7,4);
\draw [dashed] (0,0) -- (-1,0);
\draw [dashed] (0,4) -- (-1,4);
\draw (6,0) arc (-45:45:2.828);
\draw (0.2,-0.3) node {$x_j$};
\fill (0,0) circle (3pt);
\draw (2.2,-0.3) node {$x_{j-1}$};
\fill (2,0) circle (3pt);
\draw (5.2,-0.3) node {$x_{j-2}$};
\fill (5,0) circle (3pt);
\draw (6.2,-0.3) node {$x_{j-3}$};
\fill (6,0) circle (3pt);
\draw (0.2,4.3) node {$x_0$};
\fill (0,4) circle (3pt);
\draw (2.2,4.3) node {$x_1$};
\fill (2,4) circle (3pt);
\draw (5.2,4.3) node {$x_2$};
\fill (5,4) circle (3pt);
\draw (6.2,4.3) node {$x_3$};
\fill (6,4) circle (3pt);
\draw (0.7,2) node {$y_1$};
\fill (1,2) circle (3pt);
\draw (3.3,2) node {$y_2$};
\fill (3,2) circle (3pt);
\draw (4.1,0.7) node {$y_4$};
\fill (4,1) circle (3pt);
\draw (4.1,3.3) node {$y_3$};
\fill (4,3) circle (3pt);
\draw (4.7,2) node {$y_5$};
\fill (5,2) circle (3pt);
\draw (6.3,2) node {$y_6$};
\fill (6,2) circle (3pt);
\end{tikzpicture}
\end{center}
\caption{Contradiction for Case 1, $k=5$, $d_2=3$.}
\end{figure}
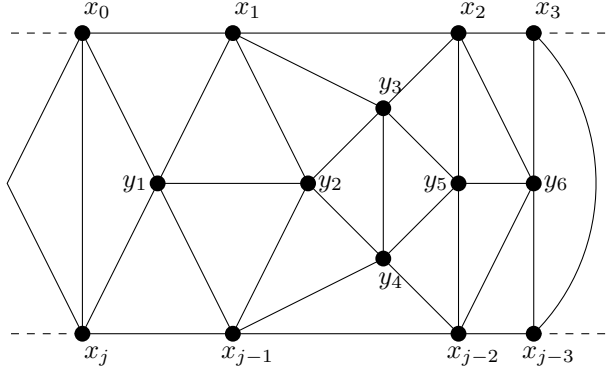

The two neighbors of $x_0$ and $x_j$ in $R_1$ must be the same, to produce a triangle bordering $x_0 x_j$.  Call this neighbor $y_1$.  The paths $y_1 x_0 x_1$ and $y_1 x_j x_{j-1}$ must close to create faces.  We cannot have $x_1 = x_{j-1}$, else $y_1$ would root inflorescences in one or both of these triangles.  Thus, $y_1$ has one additional neighbor, say $y_2$.

Now $x_1$ and $x_{j-1}$ each have three additional neighbors, one of which must be $y_2$ as the faces bordered by their edge with $y_1$ must close.  Now $y_2$ must have two addtional neighbors, one of which must be the neighbor of $x_1$ and the other of $x_{j-1}$ along the edge incident to these vertices which is nearest to $y_2$ and on the other side from $y_1$.  These must be two distinct neighbors, else $y_2$ needs another neighbor (say $z$) inside one or the other of the two resulting faces; the putative $z$ can then have only at most three neighbors on the boundary of the face and requires additional neighbors within the face, which lack sufficient boundary vertices to connect to and thus form boundaries of faces of degree greater than 3.

Let the new neighbors of $y_2$ be $y_3$ and $y_4$.  They must be adjacent.  Since $x_1$ and $x_{j-1}$ need an additional neighbor outside the faces containing their edge with $y_2$, we must have an $x_2$ and $x_{j-2}$ on the outer face; these cannot be equal, for if they were, $x_2$ would need to be adjacent to all four of $x_1$, $x_{j-1}$, $y_3$, and $y_4$, and would need an additional neighbor in one of its bounded faces, say $z$, which could be adjacent to at most three of its bounding neighbors; $z$ would need additional neighbors which would form boundaries of faces of degree greater than 3.

Now $x_2$ must be adjacent to $y_3$ and $x_{j-2}$ to $y_4$.  Further, $y_3$ and $y_4$ require an additional neighbor each, not within any of their so-far closed faces (it would be unable to connect sufficiently).  To form a triangle, it must be the same vertex, say $y_5$.  Now $x_2$ and $x_{j-2}$ must both be adjacent to $y_5$.

One more neighbor of $y_5$ is needed, as usual not in any of its so far closed nearby faces; call it $y_6$.  We will have $x_2$ and $x_{j-2}$ both adjacent to $y_6$, and requiring one more neighbor each, say $x_3$ and $x_{j-3}$, which cannot be the same neighbor: $y_6$ needs two more neighbors, and the extra neighbor would be create a face of degree too high.

But now $y_6$ already has five edges and thus $x_3$ and $x_{j-3}$ must be adjacent to close the relevant face.  But $x_3$ and $x_{j-3}$ still need two more neighbors each to be of degree 5, which cannot appear within any of the so far completed faces, so this edge cannot be an edge of $F$.  This contradicts our minimal choice of $x_0 x_j$.

\phantom{.}

\noindent \textbf{Case 2:} The logic is extremely similar.  Using $x_0$ as the vertex with its two additional neighbors in $R_1$ and $x_j$ with one additional neighbor in each $R_i$, we illustrate the required vertices and eventual contradiction to minimality in Figure 4.


\begin{figure}\label{k5trianglecase2}
\begin{center}
\begin{tikzpicture}
\draw (0,0) -- (0,3) -- (4,3) -- (4,0) -- (0,0) -- (1,1) -- (1.5,0) -- (3,1) -- (4,0) -- (3,2) -- (3,1) -- (2,2) -- (2,1) -- (1,2) -- (0,3) -- (1,1) -- (2,1);
\draw (1,1) -- (1,2) -- (1.5,3) -- (2,2) -- (3,2) -- (4,3);
\draw (1,2) -- (2,2);
\draw (1.5,3) -- (3,2);
\draw (1.5,0) -- (2,1) -- (3,1);
\draw (0,0) -- (-1,3) -- (0,3);
\draw [dashed] (0,0) -- (-1,0);
\draw [dashed] (-1,3) -- (-2,3);
\draw [dashed] (-1,3) -- (-2,2);
\draw [dashed] (-1,3) -- (-2,1);
\draw [dashed] (4,0) -- (4.5,0);
\draw [dashed] (4,3) -- (4.5,3);
\draw [dashed] (4,3) -- (4.5,2);
\draw (0.2,-0.3) node {$x_j$};
\fill (0,0) circle (3pt);
\draw (1.7,-0.3) node {$x_{j-1}$};
\fill (1.5,0) circle (3pt);
\draw (4.2,-0.3) node {$x_{j-2}$};
\fill (4,0) circle (3pt);
\draw (-0.8,3.3) node {$x_{d-1}$};
\fill (-1,3) circle (3pt);
\draw (0.2,3.3) node {$x_0$};
\fill (0,3) circle (3pt);
\draw (1.7,3.3) node {$x_1$};
\fill (1.5,3) circle (3pt);
\draw (4.2,3.3) node {$x_2$};
\fill (4,3) circle (3pt);
\draw (1,0.6) node {$y_2$};
\fill (1,1) circle (3pt);
\draw (0.95,2.4) node {$y_1$};
\fill (1,2) circle (3pt);
\draw (2.2,2.3) node {$y_3$};
\fill (2,2) circle (3pt);
\draw (2.2,0.7) node {$y_4$};
\fill (2,1) circle (3pt);
\draw (3,2.3) node {$y_5$};
\fill (3,2) circle (3pt);
\draw (3,0.7) node {$y_6$};
\fill (3,1) circle (3pt);
\end{tikzpicture}
\end{center}
\caption{Contradiction for Case 2, $k=5$, $d_2=3$.}
\end{figure}
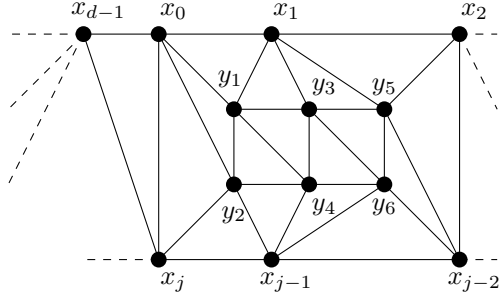

This completes the proof of the lemma.
\end{proof}

We now prove our main theorem.

\phantom{.}

\noindent \textbf{Proof of Theorem 4}

\phantom{.}

\noindent {\boldmath{$k=3,d_2=3$}} \quad\\

In this situation after substitution in Equation~\ref{v when f_1=1} we obtain
\begin{equation}\label{case 1 v}
v=2\left(\frac{d_1+3}{3}\right).
\end{equation}

Either the graph has vertices other than those that form the boundary of $F$ the exceptional face, or it does not.  If it does not, then $v = d_1 = 6$ and $e=3v/2=9$. Hence there is a chord to $F$ contrary to Lemma~\ref{nochord}.

Thus the graph must have a vertex interior to $F$. Then
\[
v=2\left(\frac{d_1+3}{3}\right) \geq d_1 +1.
\]
But then $d_1 \leq 3$, a contradiction, because $d_1 \neq d_2$.

\phantom{.}

\noindent {\boldmath{$k=3,d_2=4$}} \quad\\

In this situation after substitution in Equation~\ref{v when f_1=1} we obtain
\begin{equation}\label{case 2 v}
v=d_1+4
\end{equation}
Hence there is a set $Y$ of exactly $4$ vertices interior to the face $F$. Also $d_1\geq 6$, because $d_1 \neq d_2$ and $v$ is even, because $k$ is odd.  Because there are no chords to $F$ (Lemma~\ref{nochord}) it follows that each $x_i$ on the boundary of $F$ is adjacent to some some vertex in $Y$.

Consider an edge $xx'$ incident to $F$. Let $y$, $y'$ be the vertices adjacent to $x$ and $x'$ respectively. Because $yxx'y'$ is a path of length 3, it follows that $yy'$ is an edge. Hence every edge $x_ix_{i+1}$ incident to $F$ has a mate $y_iy_{i+1}$ on $Y$. Thus because $d_1 \geq 6$, $d_2=4$ and $|Y|=4$ at least two edges on $Y$ are mated twice to edges on the boundary of $F$. Then because $k=3$, there can be no edge with ends in $Y$ incident to a doubly mated edge, contrary to the requirement that there be at least 6 mated edges.

\begin{quotation}
\hfill ``Polygamie je zavr\v{z}en\'{\i}hodn\'{a}, pokud to nen\'{i} se mnou.''\break
\smallskip
\hfill---J\'{a}ra Cimrman
\end{quotation}

\noindent {\boldmath{$k=3,d_2=5$}} \quad\\

This part is longer, so we itemize briefly the statements we will prove:

\begin{itemize}
\item The $y_i$ to which the $x_i$ are adjacent are distinct.
\item The $y_i$ are also adjacent to a set \{$z_i$\}, none of which are $y_i$ or $x_i$ and all of which are distinct.
\item The $z_i$ are not adjacent to each other, and must be adjacent to $w_i$, which are not $x_i$, $y_i$ or $z_i$.
\item There must be exactly five $w_i$ which form a face boundary, giving a contradiction.
\end{itemize}

Because $k=3$, each vertex $x_i$ is adjacent to exactly one $y_i$.  These are distinct, due to the following cases.  If $x_i$ and $x_{i+1}$ are both adjacent to $y_i$, a triangle is formed (or a face of degree greater than 5 with inflorescence from $y_i$).  If $x_i$ and $x_{i+2}$ are adjacent to $y_i$ then $x_{i+1}$ either roots an inflorescence causing a face of degree at least 6, or if connected to $y_i$ by a path not containing an edge of $\partial F$, bounds a face of degree at least 6 on one side of that path.  Finally, if $x_i$ and $x_{i+j}$ with $j \geq 3$ ($j$ minimal among such cases) are both adjacent to $y_i$, a face of degree greater than 5 is formed if $y_i$ does not connect strictly within $\partial F$ to the path within $\partial F$ connecting $x_{i+1}$ and $x_{i+j-1}$ (which must be of length at least 2 since there are no chords), while if it does, the face of degree at least 6 occurs on the opposite side, bounded in part by $x_{i-1} x_i y_i x_{i+j} x_{i+j+1}$ and the path of length at least 2 connecting $x_{i-1}$ and $x_{i+j+1}$.

Each path $y_i x_i x_{i+1} y_{i+1}$ must be part of the boundary of a face of degree 5 with a fifth vertex $z_i$.  The $z_i$ cannot be any $y_s$: first, if $z_i = y_{i+2}$ or $y_{i-1}$ a square is formed.  Suppose instead that $z_i = y_{i+j}$, with $j$ minimal in absolute value and either $j \geq 3$ or $j \leq -2$.  The arguments are the same up to sign and a shift by 1, so suppose $j \geq 3$.  Then the path $y_i z_i x_{i+j} x_{i+j-1} y_{i+j-1}$ must bound a pentagon with fifth vertex $y_{i+1}$.  But this contradicts the minimality of $j$, for we now have a $j$ one less is absolute value (which may be the previous case).

The $z_i$ must be distinct.  If $z_i = z_{i+1}$, then $y_{i+1}$ either roots an inflorescence or the vertex other than $z_i$ and $x_{i+1}$ to which $y_{i+1}$ is connected does so, while if $z_i = z_{i+j}$ with $j$ minimal and at least 2, the vertex $z_i$ is of degree at least 4.

None of the $z_i$ are adjacent to each other: if $z_i$ is adjacent to $z_{i+1}$, a triangle is formed; if to $z_{i+2}$, then $z_{i+1}$ roots an inflorescence; if to $z_{i+j}$ with $j$ minimal, $j \geq 3$, a face of degree greater than 5 is formed.

So $z_0$ is adjacent to $w_0$, $z_1$ is adjacent to $w_1 \neq w_0$ (square), and $w_1$ is adjacent to $w_0$ to close a face.  Next $z_2$ is adjacent to $w_2$, which is not $w_1$ (square) or $w_0$ ($w_1$ would root an inflorescence), and hence $w_2$ is adjacent to $w_1$.  Next $z_3$ is adjacent to $w_3$, which is not $w_2$ (square), $w_1$ (already degree 3), or $w_0$ (hexagon or greater), and so $w_3$ is adjacent to $w_2$.  Likewise $z_4$ must exist (with only three or four $x_i$, the $w_i$ would all be adjacent and no inflorescence would be possible to increase the degree of the resulting triangle or square) and be adjacent to $w_4$, which is not $w_3$ (square), $w_2$ or $w_1$ (already degree 3), or $w_0$ ($w_3$ would root an inflorescence).  Then $w_4$ is adjacent to $w_3$, and the cycle must close to form a face of degree 5 bounded by the $w_i$.  But additional $z_i$ would make a face abutting the edge $w_4 w_0$ of too large a degree. Hence $d_1 = 5$, a contradiction.

\phantom{.}

\noindent {\boldmath{$k=4,d_2=3$}} \quad\\

In this case $v=d_1+3$ and $e=2v=2d_1+6$. Hence there is a set $Y=\{y_1,y_2,y_3\}$ of exactly $3$ vertices not incident to $F$. Each vertex $x_i$ on $F$ is adjacent to two vertices in $Y$, because $k=4$ and $F$ has no chords. This accounts for $3d_1$ edges. Thus $d_1 \leq 6$. Because $d_1 \neq d_2$, we have $4\leq d_1 \leq 6$. Furthermore there are thus $e-3d_1= 6-d_1$ edges on $Y$. But  if $y_i,y_j$ are incident to $x_h$, then $y_ix_hy_j$ is a path of length 3. Hence, because $d_2=3$, it follows that $y_iy_j$ is an edge.

Suppose that $x_1$ is adjacent to $y_1$ and $y_2$.  Then $x_2$ is also adjacent to, say, $y_2$.  It must also be adjacent to another $y_i$.  If $x_2$ is also adjacent to $y_1$, then $y_3$ is either within the regions bounded by the edges on $x_1$, $x_2$, $y_1$ and $y_2$, or not.  If it is, then it may not be adjacent to $x_1$ or $x_2$, which are already of degree 4, and it is isolated from any other $x_i$, and hence has too few possible neighbors.  If $y_3$ is external to this subgraph, then either $y_1$ or $y_2$ is internal to the cycle formed by the other three and is isolated from any possible fourth neighbors.  Thus $x_2$ is adjacent to $y_2$ and thus also $y_3$, and hence $y_2 y_3$ and then further $y_3 y_1$ are edges.  Now since $x_0 y_1$ is an edge, the triangularity of faces requires that $x_0 y_3$ be an edge, and now all $y_i$ have four neighbors and no other external vertices are possible, i.e. we have constructed the octahedron.

\phantom{.}

\noindent {\boldmath{$k=5,d_2=3$}} \quad\\

The leftmost non-boundary edge of $x_i$ and the rightmost nonboundary edge of $x_{i+1}$ must meet at vertex $y_i$ to form a triangular face.  We have the following: $y_i$ cannot be any $x_j$ since the bounding face has no chords; $y_i \neq y_{i+1}$ since $x_i$ cannot be twice adjacent to the same $y_i$.  Finally we have that $y_i \neq y_{i+j}$ for $j > 1, j \neq d$, for suppose $j$ is a minimal contradiction to this claim.  Then $x_i$ and $x_{i+1}$ both have neighbors along edges intermediate between those connecting them to $y_i$ and, respectively, $y_{i-1}$ and $y_{i+1}$; call these temporarily $z_i$ and $z_{i+1}$.  Now $y_i$ must be adjacent to these $z_i$ in order to close the triangular faces partially bounded by $y_i x_i z_i$ and $y_i x_{i+1} z_{i+1}$, since $x_i$ and $x_{i+1}$ are already of degree 5. But $y_i$ is additionally a neighbor of $x_i$, $x_{i+1}$, and $x_{i+j}$ and $x_{i+j+1}$, which requires too many edges.  (The $z_k$ cannot be any $x_\ell$ since this would be a chord, and the $x_\ell$ listed are distinct since $j>1$.)  Hence all $y_i$ are distinct; $\partial F$ consists of the base edges of a series of $d$ triangles joined at their base vertices and otherwise distinct.

Each boundary vertex $x_i$ has an additional neighbor which by definition is adjacent by an edge lying between $y_i$ and $y_{i-1}$.  Let such a vertex adjacent to $x_i$ be called $z_i$.  We again claim that all $z_i$ are distinct and not equal to $y_j$ or $x_j$ for any $j$.  The $x_j$ clause is clear since this would be a chord of the boundary.

First, by definition, $z_i$ cannot be $y_{i}$ or $y_{i-1}$, as it is a separate neighbor of $x_i$.  Now suppose that $z_i$ is $y_{i+j}$, $j > 1$ and minimal among all such $j$, including with signs reversed and distances taken modulo $d$.  In that case to close the triangular faces abutted by $z_i x_i y_{i-1}$ and $z_i x_i y_{i}$ we would require edges $z_i y_{i-1}$ and $z_i y_{i}$ respectively.  This makes $y_{i+j}$ of degree 5. Now since $z_{i+j}$ is not $y_{i+j}$, to close the triangular face abutted by $y_{i} z_i x_{i+j}$ we would require $z_{i+j} = y_{i}$, contradicting the minimality of $j$ once signs are reversed.  Hence no $z_i$ can be any $y_k$.

Next, in order to close triangular faces, each $z_i$ must be adjacent to $y_i$, to close the face partially bounded by $z_i x_i y_i$, and $y_{i-1}$, to close the face partially bounded by $z_i x_i y_{i-1}$.  If $z_i = z_{i+1}$, then $y_{i}$ possesses two more neighbors, the edges for which will increase the degree of one of the faces that $y_i$ abuts beyond 3.  If $z_i = z_{i+j}$, $j > 1$, then $z_i$ would have to be of degree at least 6 since $z_i$ is adjacent to $y_{i-1}$ and $y_i$, unless $y_{i+j} = y_{i-1}$, in which case we reverse the direction of labeling and argue as before for $j=1$.  Thus, all $z_i$ are distinct and not equal to $x_k$ or $y_k$ for any $k$.

Each $y_i$ requires another neighbor outside of the triangular faces it abuts so far; call these $w_i$.  Since $y_i$ is now of degree 5, each $w_i$ is necessarily adjacent to $z_{i-1}$ and $z_i$ to close these faces, making the $z_i$ of degree 5.  But then the faces $w_i z_{i+1} w_i$ must close cyclically, and the resulting face must be triangular.  Thus in the same manner as previous arguments we are led to the contradiction that $d_1 = d_2$, i.e. we have constructed the icosahedron.

By elimination of all cases, we have concluded the theorem:

\begin{theorem}
There are no nearly platonic graphs with one disparate face.
\end{theorem}

\section{$f_1=$ 2 or 3}
We will  say that a $k$-regular simple plane graph is a $(k;d_1^{n_1} d_2^{n_2} \cdots d_t^{n_t})$-graph if it has $n_i$ faces of degree $d_i$, $i=1,2,\ldots,t$, where $f=n_1+n_2+\cdots+n_t$.

We have found fifteen families of graphs of type $(k; d_1^2 d_2^{n_2})$; interestingly, other than the cycle all seem to be related to platonic solids.  The families are indexed by the equivalent possible pairs of distinct faces of platonic solids: the general idea is that one uses those faces as the two disparate faces, and repeats a fundamental unit around a long cycle.  Prisms and antiprisms are common examples based on the cube and octahedron respectively.  (Of course, the fundamental unit may be only a fraction of the related Platonic graph.)

The cycle is trivially the $(2;n^2 d_2^0)$ graph.

The tetrahedron has only one equivalent pair of faces, since any two faces share an edge.  Cutting this edge and repeating the resulting graph results in a ``thin cycle'' which is not the skeleton of a polyhedron, because it is not connected; however, it is a $(3;(3d)^2 3^{2d})$-graph.  Its fundamental unit is

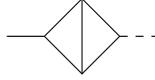
\begin{center}
\begin{tikzpicture}
\draw (0,0) -- (0.5,0.5) -- (1,0) -- (0.5,-0.5) -- (0,0);
\draw (0.5,-0.5) -- (0.5,0.5);
\draw (0,0) -- (-0.5,0);
\draw [dashed] (1,0) -- (1.5,0);
\end{tikzpicture} 
\end{center}

There are two families related to the cube.  The prisms are $(3;d^2 4^d)$-graphs isomorphic to  $C_d \square P_2$. They exist for all $d\geq 3$; the $d=4$ case is the cube.  These are polyhedral.

The other family related to the cube is the related thin cycle, with fundamental unit shown below.

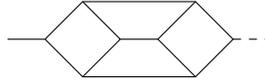
\begin{center}
\begin{tikzpicture}
\draw (0,0) -- (-0.5,0);
\draw (0,0) -- (0.5,0.5) -- (2,0.5) -- (2.5,0) -- (2,-0.5) -- (0.5,-0.5) -- (0,0);
\draw (0.5,-0.5) -- (1,0) -- (0.5,0.5);
\draw (1,0) -- (1.5,0);
\draw (2,-0.5) -- (1.5,0) -- (2,0.5);
\draw [dashed] (2.5,0) -- (3,0);
\end{tikzpicture}
\end{center}

There are three families related to the octahedron.  The antiprisms are $(4;d^2 3^{2d})$-graphs. They exist for all $d\geq 3$; the $d=3$ case is the octahedron.  They arise from choosing two opposite faces.

The thin cycle has fundamental unit shown below, related to the choice of two faces that share an edge.

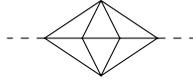
\begin{center}
\begin{tikzpicture}
\draw (0,0) -- (1.5,0) -- (0.75,0.5) -- (0.5,0) -- (0.75,-0.5) -- (1,0) -- (0.75,0.5) -- (0,0) -- (0.75,-0.5) -- (1.5,0);
\draw [dashed] (0,0) -- (-0.5,0);
\draw [dashed] (1.5,0) -- (2,0);
\end{tikzpicture}
\end{center}

One may also choose two faces in the octahedron that share only one vertex, yielding an even less polyhedral $(4;(3d)^2 3^{6d})$-graph, since the two disparate faces share multiple isolated vertices.

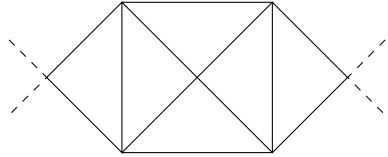
\begin{center}
\begin{tikzpicture}
\draw (0,0) -- (1,1) -- (1,-1) -- (0,0);
\draw (1,1) -- (2,0) -- (1,-1) -- (3,-1) -- (2,0) -- (3,1) -- (1,1);
\draw (3,-1) -- (3,1);
\draw [dashed] (-0.5,0.5) -- (0,0) -- (-0.5,-0.5);
\draw (3,1) -- (4, 0) -- (3,-1);
\draw [dashed] (4.5,0.5) -- (4,0) -- (4.5,-0.5);
\end{tikzpicture}
\end{center}

There are three families related to the dodecahedron.  One is prism-like, consisting of the skeleton of a truncated trapezohedron, formed by choosing two opposite faces in the dodecahedron.  These are $(3;d^2 5^{2d})$-graphs. They exist for all $d\geq 3$; the $d=5$ case is the dodecahedron.

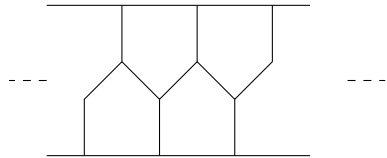
\begin{center}
\begin{tikzpicture}
\draw [dashed] (0,0) -- (-0.5,0);
\draw (0,-1) -- (3.5,-1);
\draw (0,1) -- (3.5,1);
\draw (0.5,-1) -- (0.5,-0.25) -- (1,0.25) -- (1.5,-0.25) -- (2,0.25) -- (2.5,-0.25) -- (3,0.25) -- (3,1);
\draw (1,1) -- (1,0.25);
\draw (1.5,-1) -- (1.5,-0.25);
\draw (2,1) -- (2,0.25);
\draw (2.5,-1) -- (2.5,-0.25);
\draw [dashed] (4,0) -- (4.5,0);
\end{tikzpicture}
\end{center}

The thin cycle formed from the dodecahedron by choosing two adjacent faces has the following fundamental unit.

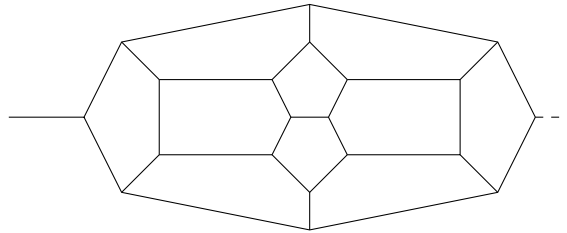
\begin{center}
\begin{tikzpicture}
\draw (-1,0) -- (0,0) -- (0.5,1) -- (3,1.5) -- (5.5,1) -- (6,0) -- (5.5,-1) -- (3,-1.5) -- (0.5,-1) -- (0,0);
\draw [dashed] (6,0) -- (6.5,0);
\draw (0.5,1) -- (1,0.5) -- (1,-0.5) -- (0.5,-1);
\draw (1,0.5) -- (2.5,0.5) -- (2.75,0) -- (2.5,-0.5) -- (1,-0.5);
\draw (2.75,0) -- (3.25,0);
\draw (2.5,0.5) -- (3,1) -- (3,1.5);
\draw (2.5,-0.5) -- (3,-1) -- (3,-1.5);
\draw (3,-1) -- (3.5,-0.5) -- (3.25,0) -- (3.5,0.5) -- (3,1);
\draw (3.5,-0.5) -- (5,-0.5) -- (5,0.5) -- (3.5,0.5);
\draw (5,-0.5) -- (5.5,-1);
\draw (5,0.5) -- (5.5,1);
\end{tikzpicture}
\end{center}

And a ``thick cycle'' formed from the dodecahedron by choosing a face and a face neither adjacent nor opposite has the following fundamental unit.

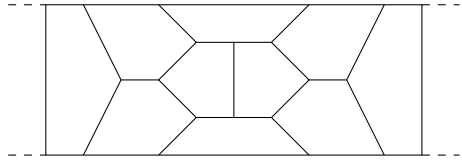
\begin{center}
\begin{tikzpicture}
\draw (0,0) -- (0,2) -- (5,2) -- (5,0) -- (0,0);
\draw (0.5,0) -- (1,1) -- (0.5,2);
\draw (1.5,0) -- (2,0.5) -- (1.5,1) -- (2,1.5) -- (1.5,2);
\draw (2.5,0.5) -- (2.5,1.5);
\draw (3.5,0) -- (3,0.5) -- (3.5,1) -- (3,1.5) -- (3.5,2);
\draw (4.5,0) -- (4,1) -- (4.5,2);
\draw (1,1) -- (1.5,1);
\draw (2,0.5) -- (3,0.5);
\draw (2,1.5) -- (3,1.5);
\draw (3.5,1) -- (4,1);
\draw [dashed] (-0.5,0) -- (0,0);
\draw [dashed] (-0.5,2) -- (0,2);
\draw [dashed] (5,0) -- (5.5,0);
\draw [dashed] (5,2) -- (5.5,2);
\end{tikzpicture}
\end{center}

Finally, there are five families related to the icosahedron.

By choosing two faces sharing a side, we obtain the following fundamental unit for the related thin cycle:

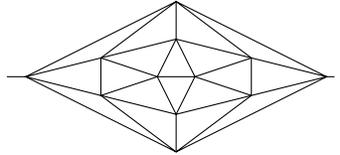
\begin{center}
\begin{tikzpicture}[scale=0.5]
\draw (-0.5,0) -- (0,0) -- (4,2) -- (8,0) -- (4,-2) -- (0,0) -- (2,0.5) -- (4,2) -- (6,0.5) -- (8,0) -- (6,-0.5) -- (4,-2) -- (2,-0.5) -- (0,0);
\draw (2,-0.5) -- (2,0.5) -- (4,1) -- (6,0.5) -- (6,-0.5) -- (4,-1) -- (2,-0.5) -- (3.5,0) -- (4,1) -- (4.5,0) -- (4,-1) -- (3.5,0) -- (4.5,0) -- (6,-0.5);
\draw (2,0.5) -- (3.5,0);
\draw (4.5,0) -- (6,0.5);
\draw (4,1) -- (4,2);
\draw (4,-1) -- (4,-2);
\draw [dashed] (8,0) -- (8.5,0);
\end{tikzpicture}
\end{center}

By choosing two faces sharing exactly one vertex, we obtain the following fundamental unit.

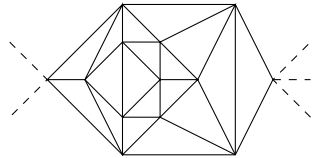
\begin{center}
\begin{tikzpicture}
\draw (0,0) -- (1,1) -- (2.5,1) -- (2.5,-1) -- (1,-1) -- (0,0) -- (0.5,0) -- (1,1) -- (1,0.5) -- (0.5,0) -- (1,-0.5) -- (1,0.5) -- (1.5,0.5) -- (2.5,1) -- (2,0) -- (1.5,0.5) -- (1.5,0) -- (2,0) -- (1.5,-0.5) -- (1,-1) -- (1,-0.5) -- (1.5,-0.5) -- (1.5,0) -- (1,-0.5);
\draw (0.5,0) -- (1,-1);
\draw (1,0.5) -- (1.5,0);
\draw (1.5,-0.5) -- (2.5,-1);
\draw (2,0) -- (2.5,-1);
\draw (1,1) -- (1.5,0.5);
\draw [dashed] (-0.5,0.5) -- (0,0) -- (-0.5,-0.5);
\draw (2.5,1) -- (3,0) -- (2.5,-1);
\draw [dashed] (3.5,-0.5) -- (3,0) -- (3.5,0.5);
\draw [dashed] (3,0) -- (3.5,0);
\end{tikzpicture}
\end{center}

Choosing one face, and one of the three faces that shares a side with the face opposite the first, gives another thick cycle, yielding a $(5;{(3d)}^2 3^{18d})$-graph:

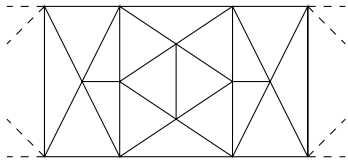
\begin{center}
\begin{tikzpicture}[scale=0.5]
\draw (0,0) -- (0,4) -- (1,2) -- (0,0) -- (2,0) -- (1,2) -- (2,4) -- (2,2) -- (2,0) -- (3.5,1) -- (2,2) -- (3.5,3) -- (2,4) -- (5,4) -- (3.5,3) -- (5,2) -- (3.5,1) -- (5,0) -- (5,4) -- (7,4) --(7,2) -- (7,0) -- (2,0);
\draw (0,0) -- (0,4) -- (2,4);
\draw (7,0) -- (7,4) -- (6,2) -- (7,0);
\draw (1,2) -- (2,2);
\draw (3.5,1) -- (3.5,3);
\draw (5,0) -- (6,2) -- (5,4);
\draw (5,2) -- (6,2);
\draw [dashed] (7,0) -- (8,0);
\draw [dashed] (7,4) -- (8,4);
\draw [dashed] (0,0) -- (-1,0);
\draw [dashed] (0,4) -- (-1,4);
\draw [dashed] (7,0) -- (8,1);
\draw [dashed] (7,4) -- (8,3);
\draw [dashed] (0,0) -- (-1,1);
\draw [dashed] (0,4) -- (-1,3);
\draw (7,0) -- (7,4);
\end{tikzpicture}
\end{center}

Choosing one face, and one of the six faces on the far side that shares a single vertex with the face opposite the first, yields the following fundamental unit:

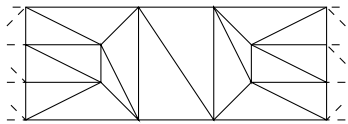
\begin{center}
\begin{tikzpicture}[scale=0.5]
\draw (0,0) -- (0,3) -- (8,3) -- (8,0) -- (0,0) -- (2,1) -- (3,0) -- (2,2) -- (2,1) -- (0,2) -- (2,2) -- (3,3) -- (5,0) -- (5,3) -- (6,1) -- (6,2) -- (8,1) -- (6,1) -- (8,0);
\draw (0,1) -- (2,1);
\draw (0,3) -- (2,2);
\draw (3,0) -- (3,3);
\draw (5,0) -- (6,1);
\draw (5,3) -- (6,2) -- (8,2);
\draw (6,2) -- (8,3);
\draw [dashed] (-0.5,0) -- (0,0) -- (-0.5,0.5);
\draw [dashed] (-0.5,1) -- (0,1) -- (-0.5,1.5);
\draw [dashed] (-0.5,2) -- (0,2);
\draw [dashed] (-0.5,2.5) -- (0,3) -- (-0.5,3);
\draw [dashed] (8.5,0) -- (8,0) -- (8.5,0.5);
\draw [dashed] (8.5,1) -- (8,1);
\draw [dashed] (8.5,1.5) -- (8,2) -- (8.5,2);
\draw [dashed] (8.5,2.5) -- (8,3) -- (8.5,3);
\end{tikzpicture}
\end{center}

Finally, choosing two opposite faces yields the following graph, which is the only unit where the two faces are separated by a path of minimum length 2.  Like the previous graphs of this type, it may be divided into a smaller repeatable fraction, in this case one third:

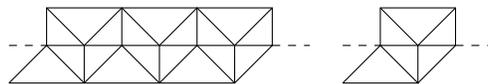
\begin{center}
\begin{tabular}{cc}
\begin{tikzpicture}[scale=0.5]
\draw (0,0) -- (6,0) -- (7,1) -- (7,2) -- (1,2) -- (1,1) -- (7,1);
\draw (0,0) -- (1,1) -- (2,0) -- (3,1) -- (4,0) -- (5,1) -- (6,0);
\draw (7,2) -- (6,1) -- (5,2) -- (4,1) -- (3,2) -- (2,1) -- (1,2);
\draw (2,0) -- (2,1);
\draw (3,2) -- (3,1);
\draw (4,0) -- (4,1);
\draw (5,2) -- (5,1);
\draw (6,0) -- (6,1);
\draw [dashed] (0,1) -- (1,1);
\draw [dashed] (7,1) -- (8,1);
\end{tikzpicture}
&
\begin{tikzpicture}[scale=0.5]
\draw [dashed] (0,1) -- (1,1);
\draw (0,0) -- (2,0) -- (3,1) -- (3,2) -- (1,2) -- (1,1) -- (2,0) -- (2,1) -- (1,1) -- (0,0);
\draw (1,2) -- (2,1) -- (3,2);
\draw (2,1) -- (3,1);
\draw [dashed] (3,1) -- (4,1);
\end{tikzpicture}
\end{tabular}

\end{center}

In all fifteen of these families, one property is constant: both of the disparate faces have the same degree, since we produce the families by repeating a given fundamental unit around a cycle, and the units involved are axially symmetric.  As of this writing, we have been unable to generate a counterexample to the following conjecture:

\begin{conjecture} If a graph is vertex-regular and planar, and all but 2 faces are of one degree, then the remaining two faces must have the same degree as each other.
\end{conjecture}

Another observation is that in all these families the longest path between the boundaries of the two disparate faces is at most two edges.  Can the distance be increased indefinitely?  Our suspicion is not.  In fact, both of these claims would be implied by a much stronger conjecture:

\begin{conjecture} The families listed above are the only types of planar graph with exactly two disparate faces.
\end{conjecture}

When there are 3 or more disparate faces the disparate face degrees may be different.  Indeed, it is possible to produce graphs with all three disparate faces having differing face degrees:

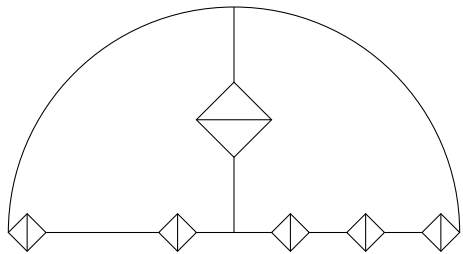
\begin{center}
\begin{tikzpicture}
\draw (-3,0) arc (180:0:3);
\draw (0,3) -- (0,2) -- (-0.5,1.5) -- (0,1) -- (0.5,1.5) -- (0,2);
\draw (-0.5,1.5) -- (0.5,1.5);
\draw (0,0) -- (0,1);
\draw (0,0) -- (0.5,0) -- (0.75,0.25) -- (1,0) -- (0.75,-0.25) -- (0.5,0);
\draw (0.75,-0.25) -- (0.75,0.25);
\draw (1,0) -- (1.5,0) -- (1.75,0.25) -- (2,0) -- (1.75,-0.25) -- (1.5,0);
\draw (1.75,-0.25) -- (1.75,0.25);
\draw (2,0) -- (2.5,0) -- (2.75,0.25) -- (3,0) -- (2.75,-0.25) -- (2.5,0);
\draw (2.75,-0.25) -- (2.75,0.25);
\draw (0,0) -- (-0.5,0) -- (-0.75,0.25) -- (-1,0) -- (-0.75,-0.25) -- (-0.5,0);
\draw (-0.75,-0.25) -- (-0.75,0.25);
\draw (-1,0) -- (-2.5,0) -- (-2.75,0.25) -- (-3,0) -- (-2.75,-0.25) -- (-2.5,0);
\draw (-2.75,-0.25) -- (-2.75,0.25);
\end{tikzpicture}
\end{center}

Of course, there are also 3-disparate graphs which display symmetries:

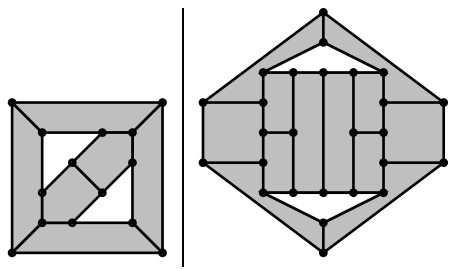
\begin{center}
\begin{tabular}{c|c}
\begin{tikzpicture}[scale=0.4, line width=1]
\foreach \i in {0,...,5}
{
 \foreach \j in {0,...,5}
 {
  \coordinate (X\i\j) at (\i,\j);
 }
}
\coordinate (X13) at (1,2);
\coordinate (X24) at (3,4);
\coordinate (X31) at (2,1);
\coordinate (X42) at (4,3);
\draw[fill=lightgray](X00)--(X11)--(X14)--(X05)--cycle;
\draw[fill=lightgray](X05)--(X55)--(X44)--(X14)--cycle;
\draw[fill=lightgray](X55)--(X50)--(X41)--(X44)--cycle;
\draw[fill=lightgray](X50)--(X41)--(X11)--(X00)--cycle;

\draw[fill=lightgray](X11)--(X13)--(X23)--(X32)--(X31)--cycle;
\draw[fill=lightgray](X44)--(X24)--(X23)--(X32)--(X42)--cycle;
\draw (X00) circle(0.1); \draw (X05) circle(0.1);
\draw (X11) circle(0.1); \draw (X13) circle(0.1);
\draw (X14) circle(0.1); \draw (X23) circle(0.1);
\draw (X24) circle(0.1); \draw (X31) circle(0.1);
\draw (X32) circle(0.1); \draw (X41) circle(0.1);
\draw (X42) circle(0.1); \draw (X44) circle(0.1);
\draw (X50) circle(0.1); \draw (X55) circle(0.1);

\end{tikzpicture}

& 

\begin{tikzpicture}[scale=0.4, line width=1]
\draw[fill=lightgray]
  (3,0)--(-1,3)--(-1,5)--(3,8)--(7,5)--(7,3)--cycle;
\draw
  (3,1)--(1,2)--(1,3)--(1,4)--(1,5)--(1,6)--(3,7)
   --(5,6)--(5,5)--(5,4)--(5,3)--(5,2)--cycle
  (3,2)--(2,2)--(2,4)--(2,6)--(3,6)--(4,6)--(4,4)
   --(4,2)--cycle;
\draw(1,6)--(2,6) (4,6)--(5,6);
\draw(1,4)--(2,4) (4,4)--(5,4);
\draw(1,2)--(2,2) (4,2)--(5,2);
\draw(3,2)--(3,6);
\draw(-1,5)--(1,5) (5,5)--(7,5);
\draw(-1,3)--(1,3) (5,3)--(7,3);
\draw(3,0)--(3,1);
\draw(3,7)--(3,8);
\draw[fill=white]
  (1,2)--(3,1)--(5,2)--cycle
  (1,6)--(3,7)--(5,6)--cycle;
\draw[fill] (-1,3) circle(0.1);
\draw[fill] (-1,5) circle(0.1);
\draw[fill] (1,2) circle(0.1);
\draw[fill] (1,3) circle(0.1);
\draw[fill] (1,4) circle(0.1);
\draw[fill] (1,5) circle(0.1);
\draw[fill] (1,6) circle(0.1);
\draw[fill] (2,2) circle(0.1);
\draw[fill] (2,4) circle(0.1);
\draw[fill] (2,6) circle(0.1);
\draw[fill] (3,0) circle(0.1);
\draw[fill] (3,1) circle(0.1);
\draw[fill] (3,2) circle(0.1);
\draw[fill] (3,6) circle(0.1);
\draw[fill] (3,7) circle(0.1);
\draw[fill] (3,8) circle(0.1);
\draw[fill] (4,2) circle(0.1);
\draw[fill] (4,4) circle(0.1);
\draw[fill] (4,6) circle(0.1);
\draw[fill] (5,2) circle(0.1);
\draw[fill] (5,3) circle(0.1);
\draw[fill] (5,4) circle(0.1);
\draw[fill] (5,5) circle(0.1);
\draw[fill] (5,6) circle(0.1);
\draw[fill] (7,3) circle(0.1);
\draw[fill] (7,5) circle(0.1);
\end{tikzpicture}
\end{tabular}
\end{center}

Our concern in this paper is with the restricted cases, and so we do not delve into these graph types.  It is intuitively obvious that as $d_1$ grows, construction of a $(k;f_1^{d_1} f_2^{d_2})$-graph becomes easier.  It might be of interest to graph theorists to make this intuition more rigorous by means of some statistic on the set of planar graphs.

\end{document}